\documentclass[11pt]{amsart}
\usepackage{amssymb,amsmath}
\usepackage[mathscr]{eucal}
\usepackage{xcolor}
\usepackage{hyperref}

\theoremstyle{plain}
\newtheorem{prop}{Proposition}[section]
\newtheorem{thm}[prop]{Theorem}
\newtheorem{cor}[prop]{Corollary}
\newtheorem{lem}[prop]{Lemma}

\theoremstyle{definition}
\newtheorem{dfn}[prop]{Definition}
\newtheorem{rem}[prop]{Remark}
\newtheorem{war}[prop]{Warning}
\newtheorem{rems}[prop]{Remarks}
\newtheorem{example}[prop]{Example}
\newtheorem{lab}[prop]{}
\newtheorem{schol}[prop]{Scholium}

\newcommand{\A}{\mathbb A}
\newcommand{\N}{\mathbb N}
\newcommand{\Q}{\mathbb Q}
\newcommand{\R}{\mathbb R}
\newcommand{\Z}{\mathbb Z}

\newcommand{\m}{\mathfrak m}

\newcommand{\x}{{\mathtt{x}}}

\DeclareMathOperator{\Spec}{Spec}

\newcommand{\Hom}{{\rm Hom}}
\newcommand{\supp}{{\rm supp}}

\newcommand{\du}{{\scriptscriptstyle\vee}}

\renewcommand{\iff}{\Leftrightarrow}
\renewcommand{\epsilon}{\varepsilon}

\newcommand{\To}{\Rightarrow}

\newcommand{\isoto}{\overset{\sim}{\to}}
\newcommand{\into}{\hookrightarrow}

\newcommand{\ol}[1]{\overline{#1}}

\renewcommand{\setminus}{\smallsetminus}
\newcommand{\plus}{{\scriptscriptstyle+}}
\newcommand{\all}{\forall\>}
\newcommand{\ex}{\exists\>}
\newcommand{\wh}[1]{\widehat{#1}}

\newcommand\PO{\text{PO}}
\newcommand\QM{\text{QM}}

\renewcommand{\choose}[2]{\genfrac(){0pt}{}{#1}{#2}}

\begin{document}

\title
[Pure states, nonnegative polynomials and sums of squares]
{Pure states, nonnegative polynomials\\and sums of squares}

\author{Sabine Burgdorf}
\address{Institut de Recherche Math\'ematique de Rennes,
         Universit\'e de Rennes I,
         Campus de Beaulieu,
         35042 Rennes cedex,
         France}
\email{sabine.burgdorf@univ-rennes1.fr}
\author{Claus Scheiderer}
\address{Universit\"at Konstanz,
         Fachbereich Mathematik und Statistik,
         78457 Konstanz,
         Allemagne}
\email{claus.scheiderer@uni-konstanz.de}
\author{Markus Schweighofer}
\address{Institut de Recherche Math\'ematique de Rennes,
         Universit\'e de Rennes I,
         Campus de Beaulieu,
         35042 Rennes cedex,
         France}
\email{markus.schweighofer@univ-rennes1.fr}
\dedicatory{Professor Alexander Prestel gewidmet aus Anlass seiner
  Emeritierung}
\date{May 26, 2009}

\begin{abstract}
In recent years, much work has been devoted to a systematic study of
polynomial identities certifying strict or non-strict positivity of
a polynomial $f$ on a basic closed set $K\subset\R^n$. The interest
in such identities originates not least from their importance in
polynomial optimization. The majority of the important results
requires the archimedean condition, which implies that $K$ has to be
compact. This paper introduces the technique of pure
states into commutative algebra. We show that this technique allows
an approach to most of the recent archimedean Stellens\"atze that is
considerably easier and more conceptual than the previous proofs. In
particular, we reprove and strengthen some of the most important
results from the last years. In addition, we establish several such
results which are entirely new. They are the first that allow $f$ to
have arbitrary, not necessarily discrete, zeros in $K$.
\end{abstract}

\maketitle


\section*{Introduction}

Consider a sequence $g_1,\dots,g_r\in\R[\x]=\R[x_1,\dots,x_n]$ of
real polynomials together with the basic closed semi-algebraic set
$K=\{x\colon g_1(x)\ge0$, $\dots,g_r(x)\ge0\}$ in $\R^n$. Given a
polynomial $f\in\R[\x]$ which is nonnegative on $K$, it is an
important problem, both from a theoretical and from a practical point
of view, to understand whether there exist simple algebraic
certificates that make the nonnegative character of $f$ evident.
Traditionally, a result stating the existence of a particular type of
such certificates is called a \emph{Positivstellensatz}, or a
\emph{Nichtnegativstellensatz}, depending on whether $f$ is supposed
to be strictly or only non-strictly positive.

Krivine \cite{Kr} and Stengle \cite{St} proved that such certificates
always exist. However, their results amount to rational
representations of $f$, that is, representations with denominators.
Much harder to establish, but also much more powerful when they
exist, are denominator-free representations of $f$, such as
$$f=s_0+\sum_{i=1}^rs_ig_i,\quad f=\sum_{i_1=0}^1\cdots\sum_{i_r=0}^1
s_{i_1,\dots,i_r}\cdot g_1^{i_1}\cdots g_r^{i_r}$$
or
$$f=\sum_{i_1,\dots,i_r\ge0}a_{i_1,\dots,i_r}\cdot g_1^{i_1}\cdots
g_r^{i_r},$$
in which the $s_i$ or $s_{i_1,\dots,i_r}$ are sums of squares of
polynomials and the $a_{i_1,\dots,i_r}$ are nonnegative real numbers.
The study of such identities comprises questions of existence and
complexity as well as algorithmic aspects. Considerable research
efforts have been spent in recent years on these questions (see
\cite{PD}, \cite{Ma:book}, \cite{Sch:guide}), not least because of their
central importance in polynomial optimization (see \cite{La} for an
excellent survey).

An \emph{urversion} of a denominator-free representation result is
the so-called archimedean representation theorem, due to Stone,
Krivine, Kadison, Dubois and others. See \cite{PD} Sect.~5.6, and
also Thm.\ \ref{repth} below. It asserts that $f$ has a
representation as desired, provided that $f>0$ on $K$ and the
archimedean condition holds. Many refinements of this result have
been proved in the last decade, notably extensions to cases where
$f$ is allowed to have zeros in $K$. Some of them are recalled in
Sect.~6 below. A common feature of all these results is the
archimedean hypothesis. See \ref{archrappl} for its technical
definition. Note that in any case, this condition implies that $K$ is
bounded, hence compact.

The purpose of this paper is to lay out a new approach to these
results and to new archimedean \emph{Stellens\"atze}, which is based
on pure states of the associated cones in $\R[\x]$. This new approach
permits proofs which are considerably more transparent, easy and
uniform than the existing ones. In a number of cases, we arrive at
substantially stronger results than known so far. In addition, using
the new technique, we prove several archimedean
Nichtnegativstellens\"atze which are completely new. Altogether, we
believe that this paper gives ample support to our claim that the
consequent use of pure states is a powerful tool in the study of
archimedean Stellens\"atze. We remark that the results presented here
do by far not exhaust the applications of this technique. We plan to
give further applications elsewhere.

The technique of pure states relies on an old separation theorem for
convex sets in a real vector space $V$, due to Eidelheit and Kakutani
(\cite{Ei}, \cite{Kk}). Combined with the Krein-Milman theorem, it
yields a sufficient condition for membership in a convex cone
$C\subset V$, provided that $C$ has an order unit (also known as
algebraic interior point): If $x\in V$ and all nonzero states of $C$
have strictly positive value in $x$, then $x\in C$. The first
systematic use of this criterion was probably made by Goodearl and
Handelman \cite{GH}.

The starting point for this work was a remark of Handelman made to
the third author in 2004. Handelman pointed out that a slightly weaker
version of Theorem 2 in \cite{Sw:cert} (corresponding to the special case
$M=S$ in Theorem \ref{sumbiti} below) can be proved easily by using pure states.

We now give a brief overview of the contents of this paper. Among its
seven sections, the first five are preparatory in character, while
the last two contain the main applications. After a few notational
preliminaries in Sect.~1, we recall the general Goodearl-Handelman
criterion in Sect.~2. From Sect.~3 on we work in a commutative ring
$A$ and consider (pseudo-) modules $M$ over subsemirings $S$ of $A$.
After studying order units in such $M$ in general (Sect.~3), we prove
an important fact in Sect.~4, which applies in the situations which
are most common ($S$ archimedean or $S$ containing all squares): If
$M$ contains an order unit with respect to the ideal it generates,
then the associated pure states satisfy a multiplicative law of a
very peculiar form. See Cor.\ \ref{dicho} for a summarizing
statement. This fact lies at the basis of all later applications.
Sect.~5 discusses the question whether intersecting $M$ with an ideal
of $A$ preserves the existence of an order unit. This is an important
technical point, as explained in \ref{remmoral}.

In Sect.~6 we review some of the most important Positiv- and
Nichtne\-ga\-tiv\-stellens\"atze in real algebra. Using pure states,
we reprove them in an elegant and uniform way. For some of them we
arrive at statements that are considerably stronger than previously
known (Theorems \ref{sumbiti}, \ref{lgcrit}). Finally, in Sect.~7 we
use pure states to arrive at Nichtnegativstellens\"atze which are
entirely new. The so far known results of this type apply only
(essentially) in the case where the zeros of the polynomial $f$ in
$K$ are discrete. The two main results presented here are Theorems
\ref{glatterandpkte} and \ref{innpkte}. In both, the zero set of $f$
in $K$ can have any dimension. While in Thm.\
\ref{glatterandpkte}, this zero set necessarily lies in the boundary
of $K$ (relative to its Zariski closure), Thm.\ \ref{innpkte} applies
typically when the zeros lie in the (relative) interior of $K$. A
particularly concrete case of Thm.\ \ref{glatterandpkte} is Thm.\
\ref{polytop}, dealing with polynomials nonnegative on a polytope and
vanishing on a face. It becomes visible in Theorems
\ref{glatterandpkte} and \ref{innpkte} how pure states on suitable
ideals of the polynomial ring are closely related to directional
derivatives (of order one in \ref{glatterandpkte}, of order two in
\ref{innpkte}).

In most parts of this paper, our setup is more general than real
polynomial rings and semi-algebraic sets in $\R^n$. We explain in
\ref{xplgen} why we think such a greater generality is necessary.


\section{Notations and conventions}

\begin{lab}
We start by recalling some terminology (mostly standard) from real
algebra. General references are \cite{PD}, \cite{Ma:book},
\cite{Sch:guide}.

Let $A$ be a commutative ring (always with unit), and let
$S\subset A$ be a semiring, i.\,e., a subset containing $\{0,1\}$ and
closed under addition and multiplication. A subset $M\subset A$ is
called an \emph{$S$-pseudomodule} if $0\in M$, $M+M\subset M$ and $SM
\subset M$. If in addition $1\in M$ then $M$ is said to be an
\emph{$S$-module}. The \emph{support} of $M$ is the subgroup $\supp
(M)=M\cap(-M)$ of $A$; this is an ideal of $A$ if $S-S=A$. We
sometimes write $a\le_Mb$ to express that $b-a\in M$, for $a$, $b\in
A$. The relation $\le_M$ is anti-symmetric modulo $\supp(M)$,
transitive, and compatible with addition and with multiplication by
elements of $S$.

Particularly important is the case where $S=\Sigma A^2$, the semiring
of all sums of squares in $A$. The $\Sigma A^2$- (pseudo-) modules
are called \emph{quadratic (pseudo-) modules} in $A$. A semiring
$S\subset A$ is called a \emph{preordering} in $A$ if it contains
$\Sigma A^2$. When $\frac12\in A$ we have $\Sigma A^2-\Sigma A^2=A$
by the identity $4x=(x+1)^2-(x-1)^2$, and so, in this case,
$\supp(M)$ is an ideal for every quadratic pseudomodule $M$.

Given finitely many elements $a_1,\dots,a_r\in A$, we write
$$\QM(a_1,\dots,a_r)\>:=\>\Sigma A^2+\Sigma A^2\cdot a_1+\cdots+
\Sigma A^2\cdot a_r$$
resp.
$$\PO(a_1,\dots,a_r)\>:=\>\QM\bigl(a_1^{i_1}\cdots a_r^{i_r}\colon
i_1,\dots,i_r\in\{0,1\}\bigr)$$
for the quadratic module (resp.\ the preordering) generated by $a_1,
\dots,a_r$ in~$A$.
\end{lab}

\begin{lab}\label{archrappl}
Let $M\subset A$ be an additive semigroup containing~$1$. Then $M$
is said to be \emph{archimedean} if for every $a\in A$ there is $n\in
\N$ with $a\le_Mn$. In other words, $M$ is archimedean if and only if
$A=\Z+M$.

Note that when $M$ is archimedean, every semigroup containing $M$ is
archimedean as well.
See Remark \ref{remsomu} below for examples of archimedean semigroups.
\end{lab}

\begin{war}
In the functional analytic literature, $M$ like in \ref{archrappl} is
called archimedean if no $a\in A\setminus M$ has the property that
$\N a$ has a lower bound in $A$
with respect to $\le_M$ (see, e.\,g., p.~20 in \cite{Go}).
Our definition is completely different and coincides with the usual
terminology in real algebra (see, e.\,g., 1.5.1 in \cite{Sch:guide}).
\end{war}

\begin{lab}\label{dfnxm}
Given any subset $M\subset A$, we write
$$X(M):=\bigl\{\phi\in\Hom(A,\R)\colon\phi|_M\ge0\bigr\}$$
(where $\Hom(A,\R)$ denotes the set of ring homomorphisms $A\to\R$)
and
$$Z(M):=X(M\cup-M)=\bigl\{\phi\in\Hom(A,\R)\colon\phi|_M=0\}.$$
Considering $\Hom(A,\R)$ as a subset of $\R^A=\prod_A\R$, this set
has a natural topology. When $M$ is an archimedean semigroup in $A$,
the subset $X(M)$ of $\Hom(A,\R)$ is compact.

Write $X:=\Hom(A,\R)$. Every $a\in A$ induces a continuous map
$\wh a\colon X\to\R$ by evaluation. Thus we have the canonical ring
homomorphism (not necessarily injective)
$$A\to C(X,\R),\quad a\mapsto\wh a$$
(here $C(X,\R)$ is the ring of continuous real-valued functions on
$X$). Thinking in this way of the elements of $A$ as $\R$-valued
functions, it is natural to write $a(x)$ instead of $x(a)$, for
$a\in A$ and $x\in X$, an abuse of notation that we will often
commit.
\end{lab}

\begin{schol}
Let $A$ be a finitely generated $\R$-algebra. To emphasize the
geometric point of view we will frequently identify $\Hom(A,\R)$ with
$V(\R)$, the set of $\R$-points of the affine algebraic $\R$-scheme
$V=\Spec(A)$. Thus, if $M\subset A$ is any subset, we have
$$X(M)=\{x\in V(\R)\colon\all f\in M\ f(x)\ge0\}.$$
If $M$ is finite, or a finitely generated quadratic module in $A$,
$X(M)$ is a basic closed semi-algebraic set in $V(\R)$.

Any choice of finitely many $\R$-algebra generators $a_1,\dots,a_n$
of $A$ gives an identification of $\Hom(A,\R)=V(\R)$ with a real
algebraic subset of $\R^n$, via the map
$$\Hom(A,\R)\into\R^n,\quad x\mapsto\bigl(x(a_1),\dots,x(a_n)
\bigr).$$
The image set is the zero set of the ideal of relations between
$a_1,\dots,a_n$, and hence is real algebraic. Generally it is
preferable not to fix a set of generators in advance, and only to
introduce affine coordinates when it becomes necessary.
\end{schol}

\begin{lab}\label{xplgen}
A word on the generality of our setup. Preorderings, and
more generally quadratic modules, in polynomial rings over $\R$ are
the most traditional context for positivity results (see \cite{PD},
\cite{Ma:book}, \cite{Sch:guide}). But there are also prominent
examples which do not fit this context, like theorems by P\'olya
and Handelman \cite{Ha1}, \cite{Ha2}, \cite{Sw:cert}. These are cases
where the required algebraic objects are semirings, or modules over
semirings. It is often preferable, or even necessary, to work with
arbitrary finitely generated $\R$-algebras, instead of just
polynomial rings over $\R$. Finally, we feel that applications to
rings of arithmetic nature, like finitely generated algebras over
$\Z$ or $\Q$, are interesting enough as to not exclude these cases
a~priori.

Given all this, our basic general setup will consist of a ring $A$
and an additive semigroup $M\subset A$ (with $0\in M$).
We feel free to assume $\Q\subset A$ and $\Q_\plus M\subset M$ when
this helps to simplify technical details. Usually this does not mean
much loss of generality, since one can always pass from $A$ and $M$
to $A_\Q=A\otimes\Q$ and $M_\Q=\{x\otimes\frac1n\colon n\in\N\}$.
None of the methods discussed in this paper sees a difference between
$f\in M$ and $\ex n\in\N$ $nf\in M$.
\end{lab}

\begin{lab}
By $\N=\{1,2,3,\dots\}$ we denote the set of natural numbers. The set
of nonnegative rational, resp.\ nonnegative real, numbers is written
$\Q_\plus$, resp.\ $\R_\plus$.
\end{lab}


\section{Convex cones and pure states}

\begin{lab}\label{states}
Let $G$ be an abelian group, written additively, and let $M\subset G$
be a subsemigroup (always containing~$0$). The subgroup $\supp(M):=
M\cap(-M)$ of $G$ is called the \emph{support} of~$M$. We neither
assume $\supp(M)=\{0\}$ nor $M-M=G$ in general. It is often useful to
work with the relation $\le_M$ on $G$ defined by $x\le_My$ $:\iff$
$y-x\in M$.

A group homomorphism $\varphi\colon G\to\R$ into the additive group
of reals is called a \emph{state} of $(G,M)$ if $\varphi|_M\ge0$. We
sometimes denote the convex cone of all states by $S(G,M)$.

An element $u\in M$ is called an \emph{order unit} of $(G,M)$ if $G=
M+\Z u$, or equivalently, if for every $x\in G$ there is $n\in\N$
with $x\le_Mnu$. In general, there need not exist any order unit, not
even when $G=M-M$ (which clearly is a necessary condition).
\end{lab}

\begin{example}
If $A$ is a ring and $M\subset A$ is an additive semigroup
containing~$1$, then $M$ is archimedean (see \ref{archrappl}) if and
only if $1$~is an order unit of $(A,M)$.
\end{example}

\begin{example}
A typical and frequently used example is when $G=V$ is a vector space
over $\R$ (of any dimension) and $M$ is a \emph{convex cone} in $V$,
i.\,e., $M$ is non-empty and satisfies $M+M\subset M$ and $\R_\plus M
\subset M$. The convex cone $S(V,M)$ of all states of $(V,M)$ is
equal to the \emph{dual cone}
$$M^*\>=\>\bigl\{\varphi\in V^\du\colon\varphi|_M\ge0\bigr\}$$
of $M$ (regarded as sitting in the dual linear space $V^\du$),
provided that $V=M-M$. (If $M$ does not span $V$, there exist
additive maps $V\to\R$ vanishing on $M$ which are not $\R$-linear.)

The order units of $(V,M)$ are also known under the name
\emph{algebraic interior points} of $M$ (e.\,g.\ \cite{Koe} p.~177,
\cite{Ba} III.1.6). In particular, when $\dim(V)<\infty$, the order
units of $(V,M)$ are precisely the interior points of $M$ with
respect to the euclidean topology on $V$. Hence, in this case, an
order unit exists if and only if $V=M-M$.
\end{example}

\begin{lab}
Assume that $(G,M)$ has an order unit $u$. Then every nonzero state
$\varphi$ of $(G,M)$ satisfies $\varphi(u)>0$. We say that $\varphi$
is a \emph{monic state} of $(G,M,u)$, or for brevity, simply a
\emph{state} of $(G,M,u)$,
if $\varphi(u)=1$. The set of all monic states will be denoted
$S(G,M,u)$.

The set $S(G,M,u)$ can be
regarded as a subset of the product vector space $\R^G=\prod_G\R$. As
such it is compact and convex. A state $\varphi\in S(G,M,u)$
is called a \emph{pure state} of $(G,M,u)$ if it is an extremal point
of the compact convex set $S(G,M,u)$, or equivalently, if
$2\varphi=\varphi_1+\varphi_2$ with $\varphi_1$, $\varphi_2\in
S(G,M,u)$ implies $\varphi=\varphi_1=\varphi_2$.

By the Krein-Milman theorem, the convex hull of the set of pure
states of $(G,M,u)$ is dense in $S(G,M,u)$. Using this fact together
with the Eidelheit-Kakutani separation theorem (\cite{Ei}, \cite{Kk},
see also \cite{Ba} III.1.7), one can prove the following fundamental
result. Originally it is due to
Effros, Handelman and Shen \cite{EHS} (see also Lemma 4.1 in \cite{GH}
and Theorem 4.12 in \cite{Go}).
\end{lab}

\begin{thm}\label{goodhandl}
Let $G$ be an abelian group and $M\subset G$ a semigroup in $G$ with
order unit~$u$. Let $x\in G$. If $\varphi(x)>0$ for every pure state
$\varphi$ of $(G,M,u)$, there is an integer $n\ge1$ with $nx\in M$.
\qed
\end{thm}

\begin{rems}\label{remsgoodhandl}
Let $G$ be an abelian group and $M\subset G$ a semigroup.

1.\
Let $G_\Q=G\otimes\Q$ and $M_\Q=\{x\otimes q\colon x\in M$, $q\in
\Q_\plus\}$. Then $S(G,M)=S(G_\Q,M_\Q)$ holds canonically. If $u\in
M$ is an order unit of $(G,M)$ then $u\otimes1$ is an order unit of
$(G_\Q,\,M_\Q)$ (the converse being false in general), and we have
$S(G,M,u)=S(G_\Q,M_\Q,u\otimes1)$. In this way one reduces the proof
of Theorem \ref{goodhandl} to the case where $G$ is a $\Q$-vector
space and $\Q_\plus M=M$.

2.\
In the situation of Theorem \ref{goodhandl}, $\varphi(x)>0$ holds for
every pure state of $(G,M,u)$ if and only if $\varphi(x)>0$ holds for
every $0\ne\varphi\in S(G,M)$. Therefore, the condition on $x$ in
\ref{goodhandl} is independent of the choice of a particular order
unit. As for the claim, note that the map $S(G,M,u)\to\R$, $\varphi
\mapsto\varphi(x)$ assumes its minimum since $S(G,M,u)$ is compact.
The set of minimizers is compact and convex, and hence has an
extremal point $\varphi$. One verifies that any such $\varphi$ is
also an extremal point of $S(G,M,u)$, i.\,e., a pure state of
$(G,M,u)$.
\end{rems}

\begin{cor}\label{cor2gh}
Assume that $(G,M)$ has an order unit $u$, and that $M$ satisfies
($na\in M$ $\To$ $a\in M$) for every $a\in G$ and $n\in\N$. Let
$x\in G$ with $\varphi(x)>0$ for every pure state $\varphi$ of
$(G,M,u)$. Then $x$ is an order unit of $(G,M)$.
\end{cor}

\begin{proof}
$x\in M$ by a direct application of Theorem \ref{goodhandl}, using
the assumption on $M$. Given $y\in G$, the map $\varphi\mapsto\frac
{\varphi(y)}{\varphi(x)}$ from the (compact convex) set $S(G,M,u)$
to $\R$ is continuous. Hence there is $n\in\N$ with $\bigl|\frac
{\varphi(y)}{\varphi(x)}\bigr|<n$, i.\,e., $\varphi(nx\pm y)>0$, for
every $\varphi\in S(G,M,u)$. Again from \ref{goodhandl} and the
assumption we get $nx\pm y\in M$.
\end{proof}


\section{Order units in rings and ideals}

\begin{dfn}
Let $A$ be a ring and $M\subset A$ an additive semigroup (with $0\in
M$, as always). For $u\in M$ we put
$$O(M,u):=O_A(M,u)\>:=\>\bigl\{a\in A\colon\ex n\in\N\ nu\pm a\in M
\bigr\},$$
or equivalently, $O(M,u)=\supp(M+\Z u)$.
\end{dfn}

So $O(M,u)$ consists of all elements which are bounded ``in absolute
value'' by some positive multiple of $u$, with respect to $\le_M$.

\begin{prop}\label{oprop}
Let $M$, $M_1$, $M_2$ be additive semigroups in $A$.
\begin{itemize}
\item[(a)]
Let $u\in M$. Then $O(M,u)$ is an additive subgroup of $M-M\subset A$
containing $\supp(M)+\Z u$.
\item[(b)]
$O(M_1,u_1)\cdot O(M_2,u_2)\subset O(M_1M_2,u_1u_2)$ for all $u_1
\in M_1$, $u_2\in M_2$, where $M_1M_2$ denotes the semigroup in $A$
generated by all products $x_1x_2$ with $x_i\in M_i$ ($i=1,2$).
\item[(c)]
Let $S$ be a semiring in $A$. Then $O(S,1)$ is a subring of $A$, and
$O(S,u)$ is an $O(S,1)$-submodule of $A$ for every $u\in S$.
\item[(d)]
Assume that $\frac12\in A$ and $M$ is a quadratic module. Then
$O(M,1)$ is a subring of $A$, and $O(M,u)$ is an $O(M,1)$-submodule
of $A$ for every $u\in M$ with $uM\subset M$.
\end{itemize}
\end{prop}

\begin{proof}
(a) is obvious. For the proof of (b) let $a_i\in O(M_i,u_i)$, say
$n_iu_i\pm a_i\in M_i$ with $n_i\in\N$ ($i=1,2$). From
\begin{eqnarray*}
3n_1n_2\,u_1u_2+\epsilon a_1a_2 & = & (n_1u_1+a_1)(n_2u_2+\epsilon
  a_2) \\
&& +n_1u_1(n_2u_2-\epsilon a_2)+n_2u_2(n_1u_1-a_1)
\end{eqnarray*}
for $\epsilon=\pm1$ we see $a_1a_2\in O(M_1M_2,u_1u_2)$.

(c) is an immediate consequence of (b). To prove (d) let $a\in
O(M,1)$, say $m\pm a\in M$. If $r>\frac m2$ is an integer, the
identity
$$(r-a)^2(m+a)+(r+a)^2(m-a)=2r^2m-2(2r-m)a^2$$
shows $a^2\in O(M,1)$. Given another element $b\in O(M,1)$, we get
$ab\in O(M,1)$ from $4ab=(a+b)^2-(a-b)^2$. So $O(M,1)$ is a subring
of $A$.

Now let $u\in M$ with $uM\subset M$, let $x\in O(M,u)$ and let $a\in
O(M,1)$ be as before. We have $nu\pm x\in M$ for some $n\in\N$, i.\,e.\
$\pm x\le_Mnu$. Multiplying with $a^2$ gives $\pm a^2x\le_Mna^2u$. By
what was said before there is $k\in\N$ with $a^2\le_Mk$. Using $uM
\subset M$ we conclude $a^2u\le_Mku$, and therefore $\pm a^2x\le_M
nku$. This shows $a^2\cdot O(M,u)\subset O(M,u)$ for every $a\in
O(M,1)$, and $O(M,u)$ is an $O(M,1)$-submodule of $A$.
\end{proof}

\begin{rems}\label{remsomu}
\hfil

1.\
If $M\subset A$ is a semigroup containing~$1$, then $M$ is
archime\-dean (\ref{archrappl}) if and only if $O(M,1)=A$.

2.\
More generally, let $M\subset A$ be any semigroup and $u\in M$. Then
$O(M,u)$ is the largest subgroup $B$ of $A$ containing $u$ with the
property that $u$ is an order unit of $(B,M\cap B)$.

3.\
The rings $O(M,1)$ were introduced in \cite{Sw:crelle}, in the case
where $M$ is a preordering. The fundamental result proved in
\cite{Sw:crelle} is that when $A$ is an $\R$-algebra of finite
transcendence degree~$d$ and $T\subset A$ is a preordering, then
$O(T,1)$ coincides with $H^d(A,T)$, the $d$~times iterated ring of
geometrically bounded elements. (See \emph{loc.~cit.} for precise
details.)

4.\
A special case of the just mentioned result is the celebrated theorem
of Schm\"udgen \cite{Sm}: If $A$ is a finitely generated $\R$-algebra
and $T\subset A$ is a finitely generated preordering, then $T$ is
archimedean if (and only if) the basic closed set $X(T)$ is compact.

5.\
The article \cite{JP} (see also \cite{PD} and \cite{Ma:book}) is
concerned with the question when quadratic modules are archimedean.
In general, this is much more subtle than for preorderings.

6.\
Let $K\subset\R^n$ be a nonempty compact convex polyhedron, described
by linear inequalities $g_1\ge0,\dots,g_s\ge0$. Let $S$ be the
semiring generated in the polynomial ring $\R[\x]:=\R[x_1,\dots,x_n]$
by $\R_\plus$ and $g_1,\dots,g_s$. By a classical theorem of
Minkowski (Thm.\ 5.4.5 in \cite{PD}), the cone $\R_\plus+\R_\plus g_1
+\dots+\R_\plus g_s\subset S$ contains every linear polynomial which
is nonnegative on $K$. Using compactness of $K$ it follows that
$O(S,1)$ contains all linear polynomials. Since $O(S,1)$ is a subring
of $\R[\x]$ (\ref{oprop}(c)), it follows that $S$ is archimedean.
\end{rems}

\begin{cor}
Let $S\subset A$ be a semiring and $M\subset A$ an $S$-module. Let
$I$, $J$ be ideals of $A$ such that $(I,S\cap I)$ has an order unit
$u$ and $(J,M\cap J)$ has an order unit $v$. Then $uv$ is an order
unit of $(IJ,\,M\cap IJ)$.
\end{cor}

\begin{proof}
The hypotheses say $I\subset O(S,u)$ and $J\subset O(M,v)$. By
\ref{oprop}(b) we have $IJ\subset O(M,uv)$, which is precisely what
was claimed.
\end{proof}

\begin{prop}\label{surpris}
Assume that $M$ is a pseudomodule over an archimedean semiring $S$ in
$A$. Then
$$O(M,f)\>=\>\supp(M+Af)$$
for every $f\in M$, and this is an ideal of $A$.
\end{prop}

\begin{proof}
$\supp(M+Af)$ is an ideal since it is stable under multiplication
with $S$ and since $S+\Z=A$. The inclusion $O(M,f)=\supp(M+\Z f)
\subset\supp(M+Af)$ is clear. Conversely let $g\in\supp(M+Af)$, say
$g=x+af=-y+bf$ with $x,y\in M$ and $a$, $b\in A$. Since $S$ is
archimedean, there is $n\in\N$ with $n\pm a\in S$ and $n\pm b\in S$.
Therefore $nf-g=(n-b)f+y$ and $nf+g=(n+a)f+x$ lie in $M$, which shows
$g\in O(M,f)$.
\end{proof}

Here is an equivalent formulation:

\begin{cor}\label{zeugen}
Let $M$ be a pseudomodule over an archimedean semiring in $A$, and
let $f\in M$. Then $f$ is an order unit of $(I,M\cap I)$ where $I:=
\supp(M+Af)$ (an ideal of $A$).
\end{cor}

\begin{proof}
The inclusion $I\subset O(M,f)$, which holds by \ref{surpris}, means
that $f$ is an order unit of $(I,M\cap I)$ (see \ref{remsomu}).
\end{proof}

Using the Goodearl-Handelman criterion, we can give still another
formulation:

\begin{cor}\label{moral}
Assume $\Q\subset A$. Let $S$ be an archimedean semiring in $A$ with
$\Q_\plus\subset S$, let $M$ be a pseudomodule over $S$, and let
$f\in A$ be fixed. Then $f\in M$ if and only if there exists an ideal
$I\subset A$ with $f\in I$ having the following two properties:
\begin{itemize}
\item[(1)]
$(I,M\cap I)$ has an order unit $u$;
\item[(2)]
$\varphi(f)>0$ for every pure state $\varphi$ of $(I,M\cap I,u)$.
\end{itemize}
Moreover, when $f\in M$, the ideals $I$ with the above properties are
precisely the ideals satisfying $Af\subset I\subset\supp(M+Af)$.
\end{cor}

\begin{proof}
If $I$ is an ideal containing $f$ with (1) and (2), then we get
$f\in M$ directly using \ref{goodhandl}.
Conversely assume $f\in M$.
Then $I:=\supp(M+Af)$ has the desired properties. Indeed, $f$ itself
is an order unit of $(I,M\cap I)$ (\ref{zeugen}). The last assertion
in \ref{moral} is also contained in \ref{zeugen}, cf.\ the remark in
\ref{remsomu}.
\end{proof}

\begin{rem}\label{remmoral}
Suppose we have $A$, $S$ and $M$ as before, and are given an element
$f\in A$ of which we want to prove that it lies in $M$. Corollary
\ref{moral} shows a possible way to proceed. In fact, most of the
main results of this paper will be concretizations of this corollary
in one or the other way. At this point, we would like to point out
the need of understanding the following two questions:
\begin{itemize}
\item[(Q1)]
Given an archimedean $S$-module $M$ and an ideal $I$ of $A$, when
does $(I,M\cap I)$ have an order unit $u$?
\item[(Q2)]
If $u$ is such an order unit, what are the pure states of $(I,M\cap
I,u)$?
\end{itemize}
We will address (Q1) in Sect.~\ref{sect:pust} and (Q2) in
Sect.~\ref{sect:ouintersect}.
\end{rem}

\begin{rem}\label{qmkeinezeugen}
Without the archimedean condition on $S$, a result like \ref{zeugen}
is usually far from true. This is demonstrated by the following
example: Let $M=\QM(x,\,y,\,1-x-y)$ in $A=\R[x,y]$, an archimedean
quadratic module by Proposition \ref{oprop}(d), and consider the
element $f=x$ of $M$. Then $\supp(M+Ax)=Ax=:I$, but $x$ is not an
order unit of $(I,M\cap I)$ (or equivalently, $O(M,x)$ is strictly
smaller than $I$). For example, $cx\pm xy\notin M$ for any $c\in\R$,
as is easily seen. In fact, we will show in \ref{nexoe} below that
$(I,M\cap I)$ does not have any order unit at all.
\end{rem}


\section{Pure states on rings and ideals}\label{sect:pust}

In \ref{remmoral} we have seen why it is important to have a good
understanding of the pure states of $(I,M,u)$, where $I$ is an ideal
of $A$ and $M\subset I$ is an $S$-pseudomodule over $S$ with order
unit $u$. We shall now give a satisfactory characterization in two
important cases, namely when $S$ is archimedean, or when $M$ is
archimedean and $S=\Sigma A^2$. These results are variations of a
theorem by Handelman (\cite{Ha1}, Prop.\ 1.2).
The main idea partially appears to some extent already in earlier
work, see Thm.~10 in \cite{BLP} or Thm.~15 in \cite{Kr}.

\begin{prop}\label{eigreinzust}
Let $A$ be a ring and $I\subset A$ an ideal. Let $S\subset A$ be an
archimedean semiring and $M\subset I$ an $S$-pseudomodule, and assume
that $(I,M)$ has an order unit~$u$. Then every pure state $\varphi$
of $(I,M,u)$ satisfies the following multiplicative law:
\begin{equation}\label{multregel}
\all a\in A\ \all b\in I\ \ \varphi(ab)=\varphi(au)\cdot\varphi(b).
\end{equation}
\end{prop}

\begin{lab}\label{lokaliszustd}
Before we start the proof of \ref{eigreinzust}, here are some
preparations. Let $u$ be an order unit of $(I,M)$. Given an additive
map $\varphi\colon I\to\R$, and given any $a\in A$ with $\varphi(au)
\ne0$, let $\varphi_a\colon I\to\R$ be the \emph{localization} of
$\varphi$ by~$a$, defined by
$$\varphi_a(b)\>:=\>\frac{\varphi(ab)}{\varphi(au)}\quad(b\in I).$$
Clearly, $\varphi_a$ is an additive map with $\varphi_a(u)=1$. If
$\varphi$ is a state of $(I,M)$ and $aM\subset M$,
then $\varphi_a$ is a state of $(I,M,u)$. If $a_1$, $a_2\in A$
satisfy $\varphi(a_iu)>0$ ($i=1,2$) then
\begin{equation*}\label{linkomblokzust}
\varphi(a_1u)\cdot\varphi_{a_1}+\varphi(a_2u)\cdot\varphi_{a_2}=
\varphi((a_1+a_2)u)\cdot\varphi_{a_1+a_2},
\end{equation*}
so $\varphi_{a_1+a_2}$ is a proper convex combination of $\varphi_
{a_1}$ and $\varphi_{a_2}$ in this case.
\end{lab}

\begin{lab}\label{pfeigreinzust}
\emph{Proof of Proposition \ref{eigreinzust}}:
In proving \eqref{multregel} we can assume $a\in S$ since $A=S+\Z$.
Fixing $a\in S$ there are two cases:

If $\varphi(au)=0$, we have to show $\varphi(aI)=0$. Now $aI=aM+
\Z au$, and so it is enough to prove $\varphi(aM)=0$. For any $x\in
M$ there is $n\in\N$ with $0\le_Mx\le_Mnu$, whence $0\le_Max\le_M
nau$, from which we get $\varphi(ax)=0$.

There remains the case where $\varphi(au)>0$. Since $S$ is
archimedean there is $n\in\Z$ with $a\le_Sn$. Choosing $n$ so large
that $\varphi(au)<n=\varphi(nu)$, we can consider the localized
(monic) states $\varphi_a$ and $\varphi_{n-a}$. As remarked before,
$\varphi_n=\varphi$ is a proper convex combination of the two. Since
$\varphi$ is a pure state we must have $\varphi_a=\varphi$, which is
identity \eqref{multregel}.
\qed
\end{lab}

The case $I=A$ and $u=1$ deserves special attention:

\begin{cor}\label{koreigreinzust}
Let $M$ be a module over an archimedean semiring in $A$. Then every
pure state of $(A,M,1)$ is a ring homomorphism $A\to\R$.
\qed
\end{cor}

A result similar to \ref{eigreinzust} is also true for quadratic
pseudomodules:

\begin{thm}\label{reinzustarchmod}
Let $I$ be an ideal of $A$ and $M\subset I$ a quadratic pseudomodule
with order unit~$u$ of $(I,M)$. Every pure state $\varphi$ of
$(I,M,u)$ satisfies \eqref{multregel} of \ref{eigreinzust}.
\end{thm}

The proof of \ref{reinzustarchmod} is somewhat more tricky. We need
two auxiliary lemmas:

\begin{lem}\label{taylcoeff}
For $n\in\N$ let
$$t_n(x)=\sum_{k=0}^n\choose{1/2}k\,(-x)^k,$$
the $n$-th Taylor polynomial of $\sqrt{1-x}$. Then the polynomial
$t_n(x)^2-(1-x)$ has nonnegative coefficients in $\Z\bigl[\frac12
\bigr]$.
\end{lem}

\begin{proof}
Fix $n$, and write $p_n(x):=t_n(x)^2-(1-x)=:\sum_{k\ge0}c_kx^k$. Then
$c_k=0$ for $k\le n$ or $k>2n$, while
$$c_k=(-1)^k\sum_{i=k-n}^n\choose{1/2}i\,\choose{1/2}{k-i}$$
for $n<k\le2n$. The term with index~$i$ in the sum has
sign $(-1)^{i-1}\cdot(-1)^{k-i-1}=(-1)^k$. This implies the lemma.
\end{proof}

\begin{lem}\label{burglem}
Keep the assumptions of \ref{reinzustarchmod}, assume moreover $\frac
12\in A$, and let $a\in A$ satisfy $aM\subset M$ and $(1-2a)u\in M$.
Then every state $\varphi$ of $(I,M)$ satisfies $\varphi((1-a)M)
\ge0$.
\end{lem}

\begin{proof}
Normalizing $\varphi$ we can assume that $\varphi$ is monic, i.\,e.,
$\varphi(u)=1$. By hypothesis we have $au\le_M\frac u2$, and
inductively we get $a^ku\le_M2^{-k}u$ for all $k\ge0$. Let $b\in M$.
There is $r\ge0$ with $2^ru-b\in M$. In order to show $\varphi
((1-a)b)\ge0$ we may replace $b$ by $2^{-r}b$, and may therefore
assume $u-b\in M$. We will show $\varphi((1-a)b)>-\epsilon$ for every
real number $\epsilon>0$.

Let $t_n(x)$ be the Taylor polynomial from Lemma \ref{taylcoeff}, and
write $p_n(x)=t_n(x)^2-(1-x)$. Due to the convergence of the binomial
series, there is $n\in\N$ with $p_n(\frac12)<\epsilon$.
Fix $n$ and write $p:=p_n$. According to \ref{taylcoeff} we have
$$p(x)=\sum_kc_kx^k$$
with nonnegative numbers $c_k\in\Z\bigl[\frac12\bigr]$. So $aM\subset
M$ implies $p(a)M\subset M$, and from $b\le_Mu$ we conclude $p(a)b\le
_Mp(a)u$. In particular, $\varphi(p(a)b)\le\varphi(p(a)u)$. On the
other hand,
$$\varphi\bigl(p(a)u\bigr)=\sum_kc_k\,\varphi(a^ku)\>\le\>\sum_kc_k
2^{-k}=p\Bigl(\frac12\Bigr)\><\>\epsilon.$$
We conclude
$$\varphi\bigl(t_n(a)^2b\bigr)-\varphi\bigl((1-a)b\bigr)=\varphi(p(a)b)
\>\le\>\varphi(p(a)u)\><\>\epsilon,$$
and so
$$\varphi\bigl((1-a)b\bigr)\>>\>\varphi(t_n(a)^2b)-\epsilon\>\ge\>
-\epsilon$$
since $M$ is a quadratic pseudomodule.
\end{proof}

\begin{lab}
\emph{Proof of Theorem \ref{reinzustarchmod}:}
We may pass from $A$, $I$ and $M$ to $A\otimes\Q$, $I\otimes\Q$ and
$M_\Q=\{x\otimes\frac1n\colon x\in M$, $n\in\N\}$, respectively (see
the remark in \ref{remsgoodhandl}). In particular, we may assume
$\frac12\in A$, and thus have $\Sigma A^2-\Sigma A^2=A$. Therefore it
is enough to prove identity \eqref{multregel} for $a\in\Sigma A^2$
and $b\in I$.

If $\varphi(au)=0$, one shows $\varphi(aI)=0$ as in
\ref{pfeigreinzust}. If $\varphi(au)>0$, choose $k\in\N$ with $au
\le_M2^ku$. For the proof of \eqref{multregel} we may replace $a$
with $2^{-(k+1)}a$, and can thus assume $(1-2a)u\in M$. Lemma
\ref{burglem} now shows $\varphi((1-a)M)\ge0$. As in the proof of
\ref{eigreinzust}, this makes $\varphi$ a proper convex combination of
the monic states $\varphi_a$ and $\varphi_{1-a}$. Since $\varphi$ is a
pure state we conclude $\varphi=\varphi_a$, which is the assertion of
\ref{reinzustarchmod}.
\qed
\end{lab}

The algebraic meaning of identity \ref{eigreinzust} \eqref{multregel}
is explained in the following easy lemma:

\begin{lem}\label{algbedeut}
Let $A$ be a ring, $I\subset A$ an ideal and $u\in I$. Let $k$ be a
field and $\varphi\colon I\to k$ an additive map satisfying $\varphi
(u)=1$. Equivalent conditions:
\begin{itemize}
\item[(i)]
$\all a\in A$ $\all b\in I$ \ $\varphi(ab)=\varphi(au)\cdot\varphi
(b)$;
\item[(ii)]
there is a ring homomorphism $\phi\colon A\to k$ such that $\varphi
(ab)=\phi(a)\cdot\varphi(b)$ for $a\in A$, $b\in I$.
\end{itemize}
Moreover, the homomorphism $\phi$ in (ii) is uniquely determined and
satisfies $\phi(a)=\varphi(au)$ for $a\in A$. Exactly one of the
following two alternatives holds:
\begin{itemize}
\item[(1)]
$\phi(u)\ne0$ and $\varphi(b)=\frac{\phi(b)}{\phi(u)}$ for every
$b\in I$;
\item[(2)]
$\phi(I)=0$.
\end{itemize}
\end{lem}

Note that the alternatives (1), resp.\ (2), are equivalent to
$\varphi(u^2)\ne0$, resp.\ $\varphi(u^2)=0$.

\begin{proof}
(i) $\To$ (ii)
One sees immediately that $\phi$ must satisfy $\phi(a)=\varphi(au)$
($a\in A$). It is readily checked that the so-defined $\phi$
satisfies (ii). The converse is clear as well.
Assuming that $\phi$ satisfies (ii), we have $\phi(b)=\phi(u)\cdot
\varphi(b)$ for every $b\in I$. If $\phi(u)\ne0$ then (1) holds.
Otherwise $\phi(u)=0$, and so $\phi(I)=0$.
\end{proof}

\begin{dfn}\label{dfnassringhom}
In the situation of \ref{algbedeut} we call $\phi$ the ring
homomorphism \emph{associated with} $\varphi$. We refer to the
identity $\varphi(ab)=\phi(a)\varphi(b)$ (for $a\in A$, $b\in I$) by
saying that $\varphi$ is \emph{$\phi$-linear}.
\end{dfn}

The setting described in \ref{algbedeut} is relevant to us since it
arises from pure states in ideals, see \ref{eigreinzust} and
\ref{reinzustarchmod}. In this situation the following additional
observation is important:

\begin{lem}\label{tangvectors}
Let $A$ be a ring, $I\subset A$ an ideal and $M\subset I$ an additive
semigroup. Let $u\in M$, and let $\varphi\colon I\to\R$ be a state of
$(I,M,u)$ fulfilling \eqref{multregel}. Then the associated ring
homomorphism $\phi\colon A\to\R$ satisfies $\phi\in X(T)$ where
$$T:=\{t\in A\colon tu\in M\}.$$
In particular, if $uM\subset M$ then $\phi\in X(M)$.
\end{lem}

\begin{proof}
If $t\in A$ is such that $tu\in M$, then $\phi(t)=\varphi(tu)\ge 0$.
\end{proof}

\begin{cor}[Dichotomy]\label{dicho}
Let $S$ be a semiring and $I$ an ideal in $A$, and let $M\subset I$
be an $S$-pseudomodule such that $(I,M)$ has an order unit $u$.
Assume that $S$ is either archimedean or a preordering. Given any
pure state $\varphi\colon I\to\R$ of $(I,M,u)$, precisely one of the
following two statements is true:
\begin{itemize}
\item[(I)]
$\varphi$ is a scaled ring homomorphism: There exists $\phi\in X(S)$
with $\phi(u)\ne0$ such that $\varphi=\frac1{\phi(u)}\cdot\phi|_I$.
\item[(II)]
There exists $\phi\in X(S+I)$ such that $\varphi$ is $\phi$-linear.
\end{itemize}
More precisely, ${\rm(I)}\iff\varphi(u^2)\ne0$, and ${\rm(II)}\iff
\varphi(u^2)=0$. In both cases, $\phi$ is uniquely determined. In
{\rm(I)} (resp.\ {\rm(II)}), one even has $\phi\in X(T)$ (resp.\
$\phi\in X(T+I)$) with $T$ defined as in Lemma \ref{tangvectors}.
Case {\rm(II)} can occur only when $I\ne A$.
\end{cor}

\begin{proof}
This is Prop.\ \ref{eigreinzust} (for $S$ archimedean) resp.\ Thm.\
\ref{reinzustarchmod} (for $\Sigma A^2\subset S$), combined with
\ref{algbedeut}. In both cases (I) and (II), note that $\phi$ is
necessarily the ring homomorphism associated with $\varphi$ (Def.\
\ref{dfnassringhom}), and hence is uniquely determined by $\varphi$.
So the additional information $\phi\in X(T)$ follows from Lemma
\ref{tangvectors}.
\end{proof}

Depending on $u$, the semiring $T$ can be larger than $S$. This is
sometimes useful, for example, in the proof of Thm.\ \ref{sumbiti}
below.

\begin{rem}\label{phiuposneg}
In general, both $\phi(u)>0$ and $\phi(u)<0$ are possible in case
(I), and accordingly, both $\phi\in X(M)$ and $\phi\in X(-M)$. In
many standard situations, however, the second cannot occur. For
example, when $M=N\cap I$ for some quadratic module $N$ of $A$, then
necessarily $\phi\in X(M)$ since $u^2\in M$. The same reasoning
applies when $M$ is a semiring.
\end{rem}

\begin{cor}\label{ouvsarch}
Assume $\Q\subset A$, and let $M$ be a quadratic module in $A$. If
$(A,M)$ has an order unit then $M$ is archimedean.
\end{cor}

In other words, if $(A,M)$ has an order unit, then $1$ is such an
order unit as well.

\begin{proof}
Let $u$ be an order unit of $(A,M)$. By \ref{cor2gh} it suffices to
show $\varphi(1)>0$ for every pure state $\varphi$ of $(A,M,u)$.
By \ref{dicho}, such $\varphi$ satisfies $\varphi(b)=\frac{\phi(b)}
{\phi(u)}$ ($b\in A$) for some ring homomorphism $\phi\colon A\to\R$
with $\phi(u)\ne0$. So $\varphi(1)=\frac1{\phi(u)}\ne0$, and $1\in M$
implies $\varphi(1)>0$.
\end{proof}

\begin{rem}
It is natural to wonder where there is a converse to Corollary
\ref{dicho}, in the following sense. In the situation given there,
assume that $\varphi$ is a state of $(I,M,u)$ that satisfies
the multiplicativity law \eqref{multregel} (and hence satisfies
(I) or (II) of \ref{dicho}, by Lemma \ref{algbedeut}). Does it follow
that $\varphi$ is a pure state, i.\,e.\ is extremal in $S(I,M,u)$?

It is easy to see that the answer must be no in general, at least
when $\varphi$ is of type (II): Fixing $\phi$, the $\phi$-linear
states of $(I,M,u)$ usually form a convex (compact) set of
positive dimension, so most of its elements are not extremal. For
example, when $M=\PO(x,\,y,\,1-x-y)$ in $A=\R[x,y]$ and $I=(x,y)$ is
the maximal ideal of the origin in $A$, then $u=x+y$ is an order unit
of $(I,M\cap I)$ (this is shown in \ref{oekrit} below). The
states of type (II) are the partial derivatives whose direction lies
in the closed first quadrant (up to normalization). Hence only two of
them are pure states.

However, when $\varphi$ is of type (I), then under suitable
additional side conditions on $M$ it is indeed true that $\varphi$
is necessarily pure. For example, this is so when $M=N\cap I$ for
some quadratic module $N$ in $A$:
\end{rem}

\begin{prop}\label{multstatisxtr}
Suppose $\R\subset A$.
Let $I$ be an ideal of $A$ and $M\subset I$ a quadratic pseudomodule
with $I=M-M$. We assume $a^2\in M$ for every $a\in I$. Then every
multiplicative state $\varphi\in S(I,M)$ is extremal in the cone
$S(I,M)$, i.\,e., $\varphi=\varphi_1+\varphi_2$ with $\varphi_i\in
S(I,M)$ implies $\varphi_i=c_i\varphi$ with $c_i\ge0$.
\end{prop}

By saying that $\varphi$ is multiplicative, we mean here that
$\varphi(xy)=\varphi(x)\varphi(y)$ holds for all $x,y\in M$.

When $A$ is a ring (possibly without unit) of $\R$-valued functions
on a set, the analogous result for multiplicative states of $(A,
A_\plus)$ was proved by Bonsall, Lindenstrauss and Phelps in 1966
(\cite{BLP}, Thm.~13). The same proof applies, essentially literally,
in our situation as well. Since Prop.\ \ref{multstatisxtr} and Cor.\
\ref{cor2multstatisxtr} won't be used elsewhere in this paper, we
skip over the details.
\qed

Combining Prop.\ \ref{multstatisxtr} with Thm.\ \ref{reinzustarchmod}
we conclude:

\begin{cor}\label{cor2multstatisxtr}
Suppose $\R\subset A$. Assume that $M$ is an archimedean quadratic
module in $A$. Then the pure states of $(A,M,1)$ are precisely the
elements of $X(M)$.
\qed
\end{cor}


\section{Existence of order units in ideals}\label{sect:ouintersect}

Given an archimedean $S$-module $M$ in $A$, and given an ideal $I$
of $A$, we are going to study when the cutted-down pseudomodule
$M\cap I$ has an order unit in $I$. See \ref{remmoral} for why this
is an important question.

\begin{prop}\label{oekrit}
Let $S\subset A$ be a semiring and $M\subset A$ an $S$-pseudomodule,
and let $I\subset A$ be an ideal generated by $x_1,\dots,x_n$. Assume
that one of the following two conditions holds:
\begin{itemize}
\item[(1)]
$(A,S)$ has an order unit $u$, and $x_1,\dots,x_n\in M$;
\item[(2)]
$(A,M)$ has an order unit $u$, and $x_1,\dots,x_n\in S$.
\end{itemize}
Then $v:=u(x_1+\cdots+x_n)$ is an order unit of $(I,M\cap I)$.
\end{prop}

\begin{proof}
Any $b\in I$ can be written $b=\sum_{i=1}^na_ix_i$ with $a_i\in A$
($i=1,\dots,n$). By assumption there is $k\in\N$ with $ku\pm a_i\in
S$ (1), resp.\ $ku\pm a_i\in M$ (2), for $i=1,\dots,n$. Hence
$kv\pm b=\sum_{i=1}^n(ku\pm a_i)x_i$ lies in $M$.
\end{proof}

For $(I,M\cap I)$ to have an order unit, it is obviously necessary
that $I$ is generated by elements of $M$. We see that this condition
is already sufficient in many cases:

\begin{cor}\label{cor2oekrit}
Let $M$ be a pseudomodule over some archimedean semiring $S$ in $A$.
If $I$ is any ideal in $A$ generated by finitely many elements of
$M$, then $(I,M\cap I)$ has an order unit.
\end{cor}

\begin{proof}
Indeed, this is \ref{oekrit}(1).
\end{proof}

On the contrary, when $M$ is merely an archimedean quadratic module
in $A$, there do in general exist ideals $I$, generated by finitely
many elements of $M$, such that $(I,M\cap I)$ does not have an order
unit. We shall now construct such examples within a somewhat more
general framework.

\begin{prop}\label{oekritqm}
Assume $\frac12\in A$. Let $M$ be an archimedean quadratic module in
$A$, and let $I$ be a finitely generated ideal in $A$.
\begin{itemize}
\item[(a)]
$(I^2,\,M\cap I^2)$ always has an order unit.
\item[(b)]
$(I,M\cap I)$ has an order unit if and only if $(I/I^2,\,
\ol{M\cap I})$ has an order unit.
\end{itemize}
\end{prop}

For the proof we need the following easy observation:

\begin{lem}\label{hsoekrit}
Let $G$ be an abelian group, $H\subset G$ a subgroup and $M\subset G$
a semigroup. If $(G/H,\ol M)$ and $(H,M\cap H)$ both have order
units, then $(G,M)$ has an order unit.
\end{lem}

\begin{proof}
By assumption there exists $v\in M\cap H$ with $H\subset\Z v+M$, and
there exists $u\in M$ with $G/H=\Z\bar u+\ol M$, i.\,e.\ $G=\Z u+M+H$.
Hence $G=\Z u+\Z v+M$. From $-v=-(u+v)+u$ we get $\Z v\subset\Z(u+v)
+M$, and similarly $\Z u\subset\Z(u+v)+M$. Therefore $G=\Z(u+v)+M$,
which means that $u+v$ is an order unit of $(G,M)$.
\end{proof}

\begin{proof}[Proof of \ref{oekritqm}]
The ideal $I^2$ is generated by squares since $4ab=(a+b)^2-(a-b)^2$.
Hence (a) is a particular case of \ref{oekrit}(2). Assertion (b)
follows from (a) together with Lemma \ref{hsoekrit}.
\end{proof}

\begin{rems}
\hfil

1.\
In the situation of \ref{oekritqm}, assume that $I=(b_1,\dots,b_m)$.
Then $u:=b_1^2+\cdots+b_m^2$ is an order unit of $(I^2,\,M\cap I^2)$.
Indeed, $u\pm b_ib_j$ is a sum of squares for all $i$, $j$, and so
the $b_ib_j$ lie in $O(M,u)$. Since $O(M,u)$ is an ideal in $A$
(\ref{oprop}), and since the $b_ib_j$ generate $I^2$, we have $I^2
\subset O(M,u)$.

2.\
In \ref{oekritqm}(b), the quotient $I/I^2$ can be replaced by $I/J$
for any ideal $J\subset I$ which is generated by finitely many sums
of squares.
\end{rems}

Here is a sample application.

\begin{prop}\label{oekrit2}
Assume $\frac12\in A$. Let $M=\QM(g_1,\dots,g_r,h_1,\dots,h_m)$ be
archimedean in $A$,
and let $I=(g_1,\dots,g_r)$. Assume that $I$ is $M$-convex, $I=
\sqrt I$, and that $h_1,\dots,h_m$ are not zero divisors modulo~$I$.
Then $(I,M\cap I)$ has an order unit if and only if
$$\bigl(I/I^2,\>\Sigma A^2\cdot\bar g_1+\cdots+\Sigma A^2\cdot
\bar g_r\bigr)$$
has an order unit.
\end{prop}

Recall here that $I$ is said to be $M$-convex if $I=\supp(M+I)$, or
equivalently, if $a$, $b\in M$ and $a+b\in I$ imply $a$, $b\in I$.
Yet another equivalent formulation is that $a$, $c\in I$, $b\in A$
and $a\le_Mb\le_M c$ together imply $b\in I$. This last version
explains why this property is called $M$-convexity.

\begin{proof}
This follows from Prop.\ \ref{oekritqm}(b) once we have shown
$$M\cap I\>\subset\>\Sigma A^2\cdot g_1+\cdots+\Sigma A^2\cdot g_r+
I^2.$$
To this end let $f\in M\cap I$, say
$$f=\sum_{i=1}^rs_ig_i+\sum_{j=0}^mt_jh_j$$
with $s_i$, $t_j\in\Sigma A^2$ und $h_0:=1$. Then $\sum_{j=0}^mt_j
h_j$ lies in $I$. This element is a sum of products $a^2h_j$ with
$a\in A$ and $j\in\{0,\dots,m\}$. Since $I$ is $M$-convex, all these
$a^2h_j$ lie in $I$. Moreover $a\in I$ in each case since $I=\sqrt I$
and the $h_j$ are not zero divisors mod~$I$. Therefore $\sum_jt_jh_j
\in I^2$, which proves the proposition.
\end{proof}

\begin{example}\label{nexoe}
1.\
In a geometric situation, e.\,g.\ for $A=\R[x_1,\dots,x_n]$, the
condition that $I$ is $M$-convex is satisfied, for example, when
$I$ is the full vanishing ideal of a real algebraic set $V\subset
\R^n$ for which $X(M)\cap V$ is Zariski-dense in $V$.
\smallskip

2.\
Let $A=\R[x,y]$ and $M=\QM(x,y,1-x-y)$, an archimedean quadratic
module in $A$. The ideal $I=(x)$ in $A$ is generated by an element
of $M$, but $(I,M\cap I)$ has no order unit.

Indeed, this is a particular case of Prop.\ \ref{oekrit2}: Via the
identification $\R[y]\isoto I/I^2$, $g(y)\mapsto xg(y)+I^2$, the cone
$\ol{M\cap I}=\Sigma\bar x$ in $I/I^2$ corresponds to the cone of
sums of squares in $\R[y]$. Clearly, this cone does not have an order
unit.
\end{example}


\section{First applications}

In this section we demonstrate how the approach via pure states gives
a uniform and elegant approach to many (if not most) of the important
known archimedean Stellens\"atze. Our proofs via pure states are
shorter and more conceptual than the previously known proofs. In
several cases we shall obtain versions that are considerably stronger
than previously known.

The selection of applications presented here is not exhaustive. We
plan to explain other applications elsewhere in a similar spirit.

\begin{thm}[Representation Theorem]\label{repth}
Let $M$ be a module over an archimedean semiring in $A$, and let
$f\in A$ with $f>0$ on $X(M)$. Then $nf\in M$ for some $n\in\N$.
\end{thm}

This fundamental theorem has been proved and re-discovered in many
versions over the time, by Stone, Krivine, Kadison, Dubois and
others
(see, e.\,g., \cite{Kr}, \cite{Kr2}). See \cite{PD}, Sect.~5.6, for
detailed historical remarks.

\begin{proof}
This is immediate from the criterion \ref{goodhandl}, since every
pure state of $(A,M,1)$ is an element of $X(M)$ by Corollary
\ref{koreigreinzust}.
\end{proof}

The version for archimedean quadratic modules was proved by Putinar
\cite{Pu} in the geometric case, and by Jacobi \cite{Ja} in an
abstract setting. Again we get it easily using the approach via pure
states:

\begin{thm}\label{kaddubqm}
Let $M$ be an archimedean quadratic module in $A$, and let $f\in A$
with $f>0$ on $X(M)$. Then $nf\in M$ for some $n\in\N$.
\end{thm}

\begin{proof}
The proof is the same as for Theorem \ref{repth}, up to replacing the
reference to Cor.\ \ref{koreigreinzust} by a reference to Thm.\
\ref{reinzustarchmod}.
\end{proof}

\begin{rem}
We just remind the reader that Theorems \ref{repth} and
\ref{kaddubqm} have many celebrated applications. Among the best ones
known are the Positivstellens\"atze by Schm\"udgen \cite{Sm} and by
Putinar \cite{Pu}.
\end{rem}

The following membership criterion, though more technical, played an
important role in the proofs of various Nichtnegativstellens\"atze
from the last years (see, e.\,g., \cite{Sch:guide} Sect.~3, in
particular 3.1.9):

\begin{thm}\label{sumbiti}
Let $M$ be an archimedean module over a semiring $S$ in $A$, and
assume that $S$ is either archimedean or $S$ is a preordering. Let
$f\in A$ with $f\ge0$ on $X(M)$. Suppose there is an identity $f=b_1
s_1+\cdots+b_rs_r$ with $b_i\in A$ and $s_i\in S$ such that $b_i>0$
on $Z(f)\cap X(M)$ ($i=1,\dots,r$). Then $nf\in M$ for some $n\in\N$.
\end{thm}

The first version of Thm.\ \ref{sumbiti} was given in \cite{Sch:MZ}
Prop.\ 2.5. Later it was generalized substantially in \cite{Sw:cert}
Thm.~2. The statement of Thm.\ \ref{sumbiti} above is still stronger
than the version in \cite{Sw:cert}, at least essentially so, since
the latter covered only the case $M=S$. (The slightly stronger
conclusion $f\in S$, instead of $nf\in S$ for some $n\in\N$, was
achieved in \cite{Sw:cert} under the assumption $\frac1q\in S$ for
some integer $q>1$. It seems that this cannot be proved with the pure
states method alone. Of course there is no difference when we assume
$\Q\subset A$ and $\Q_\plus\subset S$.)

Here is an easy proof of Thm.\ \ref{sumbiti} using pure states:

\begin{proof}
Consider the ideal $I:=(s_1,\dots,s_r)$ in $A$. Then $u:=s_1+\cdots+
s_r$ is an order unit of $(I,M\cap I)$ by Prop.\ \ref{oekrit}\,(2).
Let $\varphi$ be any pure state of $(I,M\cap I,u)$, and let $\phi
\colon A\to\R$ be the associated ring homomorphism (\ref{dicho}).
Clearly $uM\subset M$, which
implies $\phi\in X(M)$ (Cor.\ \ref{dicho}). We have
$\varphi(s_i)\ge0$ for $i=1,\dots,r$ and $\varphi(s_i)>0$ for at
least one~$i$ since $\sum_i\varphi(s_i)=1$. By \ref{goodhandl} it
suffices to show $\varphi(f)>0$.

First assume that $\varphi$ is of type (I) (see \ref{dicho}), so
$\varphi(f)=\frac{\phi(f)}{\phi(u)}$ with $\phi(u)\ne0$. Note that
$\phi\in X(M)$ implies $\phi(u)>0$. Also, since $f\ge0$ on $X(M)$, it
implies $\phi(f)\ge0$, whence $\varphi(f)\ge0$. Assuming $\varphi(f)
=0$ would give $\phi\in Z(f)\cap X(M)$, hence $\phi(b_i)>0$ ($i=1,
\dots,r$) by hypothesis. This would lead to a contradiction since
$\varphi(f)=\sum_i\phi(b_i)\varphi(s_i)$. So $\varphi(f)>0$ holds in
case~(I).

When $\varphi$ is of type (II) then $\phi\in X(M+I)\subset X(M+Af)=
Z(f)\cap X(M)$. So again $\phi(b_i)>0$ for $i=1,\dots,r$, and
$\varphi(f)=\sum_i\phi(b_i)\varphi(s_i)$ implies $\varphi(f)>0$.
\end{proof}

In \cite{Sch:surf} Thm 2.8, a local-global criterion was stated for
membership in a module $M$ over an archimedean preordering, in which
the local conditions referred to the ``localizations'' of $M$ with
respect to the maximal ideals of $A$. This criterion has turned out
to be quite powerful, cf.\ the applications mentioned in
\emph{loc.\,cit.}.

Using pure states it is easy to reprove this criterion, and in fact
to strengthen it further:

\begin{thm}\label{lgcrit}
Let $S$ be an archimedean semiring and $M$ an $S$-module in $A$. Let
$f\in A$. For every maximal ideal $\m$ of $A$, assume that there
exists $s\in S$ with $s\notin\m$ and $sf\in M$. Then $nf\in M$ for
some $n\in\N$.
\end{thm}

\begin{proof}
Let $I:=\supp(M+Af)$, and let $J'$ be the ideal generated by $M\cap
I$. For every maximal ideal $\m$ of $A$ there exists $s\in S$,
$s\notin\m$, with $sf\in M$, and hence $sf\in J'$.
This shows $f\in J'$. (The
argument is classical, we repeat it for the readers's convenience:
Choose finitely many $s_i\in S$ with $(s_1,\dots,s_r)=(1)$ and with
$s_if\in J'$ ($i=1,\dots,r$), then multiply an equation $\sum_ia_is_i
=1$ with $f$ to see $f\in J'$.) Hence there are finitely many
elements $x_1,\dots,x_m\in M\cap I$ with $f\in(x_1,\dots,x_m)$. Since
$I=\supp(M+Af)$, there are $y_i\in M\cap I$ with $x_i+y_i\in Af$ ($i=
1,\dots,r$). Let $J:=(x_1,\dots,x_r,y_1,\dots,y_r)$. Then $f\in J$,
and $u:=\sum_i(x_i+y_i)$ is an order unit of $(J,M\cap J)$ by
\ref{oekrit}(1). Note that $u=af$ for some $a\in A$.

Let $\varphi$ be a pure state of $(J,M\cap J,u)$, we are going to
show $\varphi(f)>0$. Let $\phi$ be the associated ring homomorphism,
so $\phi\in X(S)$ (Cor.\ \ref{dicho}). From $1=\varphi(af)=\phi(a)
\varphi(f)$ we get $\varphi(f)\ne0$. On the other hand, there exists
$s\in S$ with $\phi(s)\ne0$ (hence $\phi(s)>0$) and $sf\in M$.
So $0\le\varphi(sf)=\phi(s)\varphi
(f)$ shows $\varphi(f)\ge0$.
Altogether we get $\varphi(f)>0$, and the proof is once more
completed by an application of Theorem \ref{goodhandl}.
\end{proof}

\begin{rem}
When $M$ is a quadratic module (so we can assume that $S$ is a
preordering), the local condition is needed only for the maximal
ideals $\m\supset\supp(M)$. (If there is $a\in\supp(M)$ with
$a\notin\m$, then $af\in\supp(M)\subset M$.) For such $\m$, the
condition simply says $f\in M_\m$, where $M_\m$ is the quadratic
module generated by $M$ in $A_\m$.

When $\frac12\in A$ and $M=S$ is a preordering, and if we assume
$f\ge0$ on $X(S)$, the local condition is only needed for $\m\supset
I=\supp(S+Af)$. (The brief argument is given in the proof of
\cite{Sch:surf} Cor.\ 2.10.)
\end{rem}


\section{More applications}

We demonstrate now that the technique of pure states allows to
establish archimedean Stellens\"atze that are completely new. Given
a compact basic closed set $K\subset\R^n$ and a polynomial $f\in
\R[\x]$ with $f|_K\ge0$, all known results on denominator-free
representations of $f$ require (essentially) that the zero set of
$f$ in $K$ is discrete, i.\,e., finite. In contrast, this zero set
can be of arbitrary dimension in the two main results of this
sections, Theorems \ref{glatterandpkte} and \ref{innpkte} (see also
Thm.\ \ref{polytop}).

\begin{prop}\label{3nothinallg}
Assume $\Q\subset A$.
Let $M$ be a module over an archimedean preordering $S$ in $A$, let
$f\in A$ with $f\ge0$ on $X(M)$, and put $I:=\supp(M+Af)$ (an ideal
of $A$). Consider the following conditions:
\begin{itemize}
\item[(i)]
$f\in M$;
\item[(ii)]
$f$ lies in the ideal of $A$ generated by $M\cap I$, and for every
$\phi\in X(S+I)$ and every $\phi$-linear map $\varphi\colon I\to\R$
with $\varphi|_{M\cap I}\ge0$ one has $\varphi(f)\ge0$.
\end{itemize}
Then (ii) implies (i) if the ideal $I$ is finitely generated. The
converse (i) $\To$ (ii) holds unconditionally.
\end{prop}

\emph{Remark:}
``$\varphi(f)\ge0$'' at the end of condition (ii) is not a misprint.
However, (i) implies in fact $\varphi(f)>0$ whenever $\varphi$ is
nonzero.

\begin{proof}
$I$ is an ideal of $A$ since $SI\subset I$ and $S+\Z=A$. The
implication (i) $\To$ (ii) is trivial. We remark that $\varphi(f)>0$
holds in (ii) whenever $\varphi\ne0$. Indeed, $f$ is an order unit of
$(I,M\cap I)$ according to Cor.\ \ref{zeugen}.

Conversely assume that (ii) holds and $I$ is finitely generated. Let
$J$ be the ideal generated by $M\cap I$. Since $I=(M\cap I)+Af$,
it is clear that $I=J+Af$. So $f\in J$ implies $J=I$. Choose
generators $x_1,\dots,x_r\in M$ of $I$. There are elements
$y_i\in M\cap I$ with $x_i+y_i\in Af$ ($i=1,\dots,r$). The element
$u:=\sum_i(x_i+y_i)$ lies in $Af$, and is an order unit of $(I,M\cap
I)$ by \ref{oekrit}. Applying \ref{goodhandl}, we have to show
$\varphi(f)>0$ for every pure state $\varphi$ of $(I,M\cap I,u)$.

Given such $\varphi$, let $\phi\in X(S)$ be the associated ring
homomorphism (Cor.\ \ref{dicho}). From $u\in Af$ we see that $\varphi
(f)\ne0$.
If $\varphi$ is of type (II) then $\varphi(f)\ge0$ by the hypothesis.
Assume that $\varphi$ is of type (I), i.\,e., $\phi(u)\ne0$. From $u^2
\in M\cap I$ and $\varphi(u^2)=\phi(u)$ we see $\phi(u)>0$. For any
$x\in M$ we have $u^2x\in M\cap I$, therefore $0\le\varphi(u^2x)=
\phi(u)\phi(x)$, which implies $\phi(x)\ge0$. Hence $\phi\in X(M)$,
and so $\phi(f)\ge0$ follows from the hypothesis.
\end{proof}

\begin{rem}
At first sight it is surprising that $\varphi(f)\ge0$ in (ii) should
suffice (instead of $\varphi(f)>0$). The subtlety, however, lies in
the ideal $I$ and in the condition that $f$ should lie in the ideal
generated by $M\cap I$. In concrete situations it is often hard to
decide whether this is true. Even when $S$ is a preordering in
$\R[x_1,\dots,x_n]$ given by finitely many explicit generators,
there seems no general procedure known to produce generators for the
support ideal $\supp(S)$. For these reasons, Prop.\ \ref{3nothinallg}
seems to be mainly of theoretical interest.
\end{rem}

\begin{prop}\label{allgemkrit}
Let $A$ be an $\R$-algebra, let $S\subset A$ be a semiring and
$M\subset A$ an archimedean $S$-module. Assume that $S$ is either
archimedean or a preordering. Let $f\in A$ with $f\ge0$ on $X(M)$.
Assume there are $g_1,\dots,g_r\in S$ that vanish identically on
$Z(f)\cap X(M)$, such that the following two conditions are
satisfied:
\begin{itemize}
\item[(1)]
$f\in I:=(g_1,\dots,g_r)$;
\item[(2)]
for every $\phi\in Z(f)\cap X(M)$, the residue class $\bar f$ lies
in the interior of the cone $\R_\plus\bar g_1+\dots+\R_\plus\bar g_r
\subset I/\m_\phi I$, where $\m_\phi:=\ker(\phi)$.
\end{itemize}
Then $f\in M$.
\end{prop}

Note that $I/\m_\phi I$ is an $\R$-vector space of finite dimension,
which explains the meaning of interior in (2). It is clear how to
give a dual formulation of (2) using states.

\begin{proof}
By Prop.\ \ref{oekrit}(2), $u:=g_1+\cdots+g_r$ is an order unit of
$(I,M\cap I)$.
Note $u\in S$. Let $\varphi\colon I\to\R$ be a pure state of $(I,
M\cap I,u)$. We shall show $\varphi(f)>0$, which implies $f\in M$ by
Thm.\ \ref{goodhandl}. Let $\phi\in X(S)$ be the ring homomorphism
associated to $\varphi$. For every $x\in M$ we have $xu\in M\cap I$,
and so $0\le\varphi(xu)=\phi(x)$.
This shows $\phi\in X(M)$, and so $\phi(f)\ge0$ by hypothesis.
Moreover, there are two possibilities (Cor.\ \ref{dicho}):

1.\
If $\varphi$ is of type (I) then $\phi(u)\ne0$, and hence $\phi(u)
>0$ since $u\in S$. Assuming $\phi(f)=0$ would mean $\phi\in Z(f)
\cap X(M)$. This would imply $\phi(g_i)=0$ for all $i$, contradicting
$\phi(u)>0$. So $\phi(f)>0$, and hence $\varphi(f)=\frac{\phi(f)}
{\phi(u)}>0$.

2.\
If $\varphi$ is of type (II) then $\phi\in Z(f)\cap X(M)$.
The map
$\varphi$ is induced by a $\phi$-linear map $\bar\varphi\colon I/
\m_\phi I\to\R$ satisfying $\bar\varphi(\ol{M\cap I})\ge0$. In
particular, $\bar\varphi\ge0$ on the cone $\R_\plus g_1+\dots+
\R_\plus g_r$. Since $\bar f$ lies in the interior of this cone by
assumption (2),
we again get $\varphi(f)=\bar\varphi(\bar f)>0$.
\end{proof}

\begin{rems}\label{remsallgemkrit}
\hfil

1.\
Given $g_1,\dots,g_r\in S$ that vanish on $Z(f)\cap X(M)$, conditions
(1) and (2) in Prop.\ \ref{allgemkrit} can be effectively checked,
for example when $A$ is a polynomial ring over $\R$.
\smallbreak

2.\
In Prop.\ \ref{allgemkrit}, assume that $S$ is an archimedean
semiring and $M=S$. Then the sufficient conditions of
\ref{allgemkrit} are also necessary for $f\in S$, in the sense that
$f\in S$ implies the existence of $g_1,\dots,g_r\in S$ satisfying (1)
and (2). (One can simply take $r=1$ and $g_1=f$.)
\smallbreak

3.\
Assume we are given $S$, $M$ and $f$ as in \ref{allgemkrit}, with
$f\ge0$ on $X(M)$, and we want to prove $f\in M$ using this theorem.
In general, it is a subtle task to find a suitable ideal $I$ as in
this theorem (together with its generators), since conditions (1) and
(2) tend to work against each other: (1) asks for $I$ being large,
(2) asks for $I$ being small.
\end{rems}

Using the abstract criteria established so far, we shall now obtain
applications in geometric situations that are more concrete. In doing
so, the question arises how interpret conditions like
\ref{allgemkrit}(2) in a geometric way. Under suitable regularity
assumptions, this turns out to be possible.

First, we need the following lemma:

\begin{lem}\label{imimm2}
Let $(A,\m)$ be a regular local ring, and let $I\ne(1)$ be an ideal.
If $A/I$ is regular then for any $n\ge1$ the map
$$I^n/\m I^n\to\m^n/\m^{n+1}$$
induced by $I^n\subset\m^n$ is injective. Conversely, if this map is
injective for $n=1$, then $A/I$ is regular.
\end{lem}

\begin{proof}
Injectivity of this map for $n=1$ means that $I$ can be generated by
a subsequence $(x_1,\dots,x_d)$ of a regular parameter system of
$(A,\m)$. It is well known that this is equivalent to $A/I$ being
regular (e.\,g., \cite{Mt} Thm.\ 14.2).
Assuming that this is the case, the ideal $I^n$ is generated by the
monomials $x^\alpha=x_1^{\alpha_1}\cdots x_d^{\alpha_d}$ of degree
$|\alpha|=n$. These are linearly independent in $\m^n/\m^{n+1}$ over
$A/\m$ (\emph{loc.\,cit.}, Thm.\ 14.4), and so the map $I^n/\m I^n\to
\m^n/\m^{n+1}$ is injective as well.
\end{proof}

Here is an application of Prop.\ \ref{allgemkrit} to a geometric
situation. We write $\R[\x]:=\R[x_1,\dots,x_n]$.

\begin{thm}\label{glatterandpkte}
Let $S\subset\R[\x]$ be a semiring and $M$ an archimedean $S$-module.
Assume that $S$ is either archimedean or  a preordering. Let $f\in
\R[\x]$ with $f\ge0$ on $X(M)$, and let $V$ be the (reduced) Zariski
closure of $Z(f)\cap X(M)\subset\R^n$ in $\A^n$. Assume there are
$g_1,\dots,g_r\in S$ vanishing on $Z(f)\cap X(M)$ with
\begin{itemize}
\item[(1)]
$f\in(g_1,\dots,g_r)$;
\item[(2)]
for every $z\in Z(f)\cap X(M)$ and every $v\in\R^n$ with $D_vg_i(z)
\ge0$ ($i=1,\dots,r$) and $v\notin T_z(V)$ we have $D_vf(z)>0$.
\end{itemize}
If moreover every point $z\in Z(f)\cap X(M)$ is a nonsingular point
of $V$, then $f\in M$.
\end{thm}

Here we have written $D_vf(z)$ for the directional derivative of $f$
at $z$ in the direction $v$, i.\,e.,
$$D_vf(z)\>=\>\lim_{t\to0}\frac{f(z+tv)-f(z)}t.$$

\begin{proof}
Write $A:=\R[\x]$ and $I:=(g_1,\dots,g_r)$, and let $J$ be the
vanishing ideal of $V$ in $A$. We are going to apply Prop.\
\ref{allgemkrit}. To verify hypothesis (2) there, fix $z\in Z(f)\cap
X(M)$, and let $\m:=\m_z$ be the corresponding maximal ideal of $A$.
Note that $I\subset J\subset\m$.

We first show $I+\m^2=J+\m^2$. Assume to the contrary that the
inclusion $I+\m^2\subset J+\m^2$ is strict. Then there exists a
linear form $\psi\in(\m/\m^2)^\du$ vanishing on all residue classes
of elements of $I$, but not on all residue classes of elements of
$J$. This means that there is a vector $v\in\R^n$ with $v\notin T_z
(V)$ and with $D_vg(z)=0$ for all $g\in I$. But this contradicts
assumption (2), since we cannot have $D_{\pm v}f(z)>0$ for both
signs~$\pm$.

Next we show that the elements of $(I/\m I)^\du$ are directional
derivatives at $z$. It is enough to prove that the map $I/\m I\to\m/
\m^2$ induced by the inclusion $I\subset\m$ is injective. Since $A_\m
/JA_\m$ is a regular local ring by hypothesis, the map $J/\m J\to\m/
\m^2$ is injective (Lemma \ref{imimm2}), which means $J\cap\m^2=
\m J$. On the other hand, $I+(J\cap\m^2)=J$ by what has just been
proven. So $I+\m J=J$. By the Nakayama lemma this implies $IA_\m=
JA_\m$,
and so $I/\m I\to\m/\m^2$ is injective as desired.

Therefore, when $v$ runs through the vectors in $\R^n$ as in (2),
then $\varphi_v\colon\bar g\mapsto D_vg(z)$ ($\bar g\in I/\m I$) runs
through the nonzero elements in the dual of the cone $\R_\plus
\bar g_1+\dots+\R_\plus\bar g_r\subset I/\m I$. So we see that
condition (2) in \ref{glatterandpkte} corresponds precisely to (2) in
Prop.\ \ref{allgemkrit}. The proof is therefore complete.
\end{proof}

\begin{rems}\label{remsglrdpkte}
\hfil

1.\
For  Thm.\ \ref{glatterandpkte}, it is not necessary to work in a
polynomial ring $\R[\x]$, resp.\ in affine space $\A^n$. One could
replace $\A^n$ by any nonsingular affine $\R$-variety, if one is
willing to reformulate condition (2) properly in this setting. We
restricted to the case of the polynomial ring only to allow a less
technical formulation.
\smallbreak

2.\
Let $W$ be the Zariski closure of $X(M)$. Then the hypotheses of
Theorem \ref{glatterandpkte} imply that every point $z\in Z(f)\cap
X(M)$ is a boundary point of $X(M)$ relative to $W(\R)$, except when
$f$ vanishes identically on a neighborhood of $z$ in $X(M)$. Indeed,
otherwise $T_z(V)\subsetneqq T_z(W)$,
and there would be a neighborhood of $z$ in $W(\R)$ on which $g_1,
\dots,g_r$ are nonnegative.
Choose any $v\in T_z(W)$ with $v\notin T_z(V)$ and apply (2) to
$\pm v$ to get a contradiction.
(By $T_z(W)$ we denote the tangent space of $W$ at $z$ in $\R^n$.)
\end{rems}

Here is a particularly concrete case of Thm.\ \ref{glatterandpkte}.
Again we denote $\R[\x]=\R[x_1,\dots,x_n]$.

\begin{thm}\label{polytop}
Let $K\subset\R^n$ be a nonempty compact convex polyhedron, described
by linear inequalities $g_1\ge0,\dots,g_s\ge0$. Let $S$ be the
semiring in $\R[\x]$ generated by $\R_\plus$ and $g_1,\dots,g_s$. Let
$F$ be a face of $K$, and let $f\in\R[\x]$ satisfy $f|_F=0$ and $f|_
{K\setminus F}>0$. For every $z\in F$ and every $y\in K\setminus F$
assume $D_{y-z}f(z)>0$. Then $f\in S$.
\end{thm}

Speaking informally, the last hypothesis says that every directional
derivative of $f$ at a point of $F$ pointing into $K$ and not
tangential to $F$ should be strictly positive.

\begin{proof}
By Remark \ref{remsomu}, $S$ is archimedean. After
relabelling the $g_i$ we can assume that $g_1,\dots,g_r$ vanish
identically on $F$ while $g_{r+1},\dots,g_s$ don't, where $1\le r\le
s$. Then $I:=(g_1,\dots,g_r)$ is the full vanishing ideal of the
affine subspace $V$ spanned by $F$,
and so $f\in I$.

We are going to apply Theorem \ref{glatterandpkte} with $M=S$.
Condition (1) has just been established. In view of (2) fix $z\in F$,
and let $v\in\R^n$ with $v\notin T_z(V)$ and $D_vg_i(z)\ge0$ for
$i=1,\dots,r$.
We need to show $D_vf(z)>0$.

For this we would like to assure that $z+bv\in K$ for small $b>0$.
A~priori, this need not be the case. However, we still have some
freedom to adjust $v$. Choose $w\in\R^n$ such that $z+\epsilon w$
lies in the relative interior of $F$ for small $\epsilon>0$. Then
for every index $j\in\{r+1,\dots,s\}$ we have either $g_j(z)>0$ or
$D_wg_j(z)>0$.
Replace $v$ by $v+tw$ for large $t>0$. This doesn't change $D_va(z)$
for $a\in I$,
but in this way we can achieve $D_vg_j(z)>0$ for every $j\in\{1,
\dots,s\}$ with $g_j(z)=0$. Therefore, $z+bv\in K\setminus F$ for
small $b>0$, which means $v=c(y-z)$ for suitable $c>0$ and $y\in K
\setminus F$. From the hypothesis made on $f$ we therefore conclude
$D_vf(z)>0$.
\end{proof}

\begin{rem}
In the situation of Theorem \ref{polytop}, it was so far not even
known whether $f$ would lie in the preordering $\PO(g_1,\dots,g_r)$
except when $F$
is a face of codimension one.
(In this case, after extracting from $f$ the linear
equation for $F$ with the maximal possible power, one is left with a
polynomial which is strictly positive on $K$.)
\end{rem}

\begin{example}
Consider the simplex
$$K\>=\>\Bigl\{x\in\R^n\colon x_1\ge0,\dots,x_n\ge0,\>\sum_{i=1}^nx_i
\le1\Bigr\}$$
in $\R^n$, and let $S\subset\R[x_1,\dots,x_n]$ be the semiring
generated by $\R_\plus$ and $x_1,\dots,x_n$, $1-\sum_{i=1}^nx_i$.
Consider the face $F=K\cap\{x_1=\cdots=x_r=0\}$ of $K$ (with
$1\le r\le n$ being fixed). Given a polynomial $f$ with $f>0$ on
$K\setminus F$ and $f=0$ on $F$, we have $f\in S$ provided that
$\partial_{x_1}f,\dots,\partial_{x_r}f$ are strictly positive on
$F$.
\end{example}

While Theorem \ref{glatterandpkte} applies only in cases where the
zeros of $f$ in $X(M)$ lie on the boundary of $X(M)$ (see Remark
\ref{remsglrdpkte}), we will now mention a result which applies when
$f$ vanishes in interior points of $X(M)$.

\begin{thm}\label{innpkte}
Let $M=\QM(g_1,\dots,g_m)$ be an archimedean quadratic module in $\R
[\x]$. Let $f\in\R[\x]$ with $f\ge0$ on $X(M)$. Assume that the
(reduced) Zariski closure $V$ of $Z(f)\cap X(M)$ in $\A^n$ is a
local complete intersection.
For every point $z\in Z(f)\cap X(M)$, assume moreover:
\begin{itemize}
\item[(1)]
$z$ is a nonsingular point of $V$,
\item[(2)]
$\nabla f(z)=0$,
\item[(3)]
$D^2f(z)[v,v]>0$ for all $v\in\R^n$ with $v\notin T_z(V)$.
\end{itemize}
Then $f\in M$.
\end{thm}

Here $D^2f(p)[v,w]$ denotes the evaluation of the Hessian $D^2f(p)$
at the pair of vectors $(v,w)$.

\begin{proof}
Let $J$ be the vanishing ideal of $V$ in $\R[\x]$. We have $f\in J$
and are going to show $f\in J^2$. First fix $z\in Z(f)\cap X(M)$, let
$\m=\m_z$ be the corresponding maximal ideal of $\R[\x]$. Then $f\in
\m^2$ since $\nabla f(z)=0$. Since $V$ is a local complete
intersection, $J/J^2$ is locally free as a module over $\R[V]=\R[\x]
/J$ (e.\,g.\ \cite{H}, pp.~184--185). Since $\bar f\in\m_zJ/J^2$ for
every $z\in Z(f)\cap X(M)$, and since this set is Zariski dense in
$V$, it follows that $f\in J^2$.

By Prop.\ \ref{oekritqm}(a), $(J^2,M\cap J^2)$ has an order unit $u$.
Let $\varphi$ be a pure state of $(J^2,M\cap J^2,u)$, we shall show
$\varphi(f)>0$. If $\varphi$ is of type~(I) then, up to positive
scaling, $\varphi$ is evaluation in some point of $X(M)$ outside
$Z(f)$,
and so $\varphi(f)>0$. If $\varphi$ is of type~(II), there is a point
$z\in Z(f)\cap X(M)$ such that $\varphi$ is induced by a linear map
$\bar\varphi\colon J^2/\m J^2\to\R$, where $\m:=\m_z$. Since $z$ is a
nonsingular point of $V$, the map $J^2/\m J^2\to\m^2/\m^3$ induced by
the inclusion $J^2\subset\m^2$ is injective (Lemma \ref{imimm2}). The
inclusion $J/\m J\into\m/\m^2$ induces an inclusion of the second
symmetric powers of these vector spaces, which is $J^2/\m J^2\into
\m^2/\m^3$. The linear map $\bar\varphi$ can therefore be seen as a
positive semidefinite symmetric bilinear form on $J/\m J$. As such it
can be extended to $\m/\m^2$. This yields a linear extension $\tilde
\varphi\in(\m^2/\m^3)^\du$ of $\bar\varphi$ such that $\tilde\varphi
(\bar g^2)\ge 0$ for all $g\in\m$. Since the elements of $(\m^2/\m^3)
^\du$ are the symmetric second order differential operators at $z$,
it follows that there is a positive semidefinite symmetric matrix
$(s_{ij})$ such that $\varphi(g)=\sum_{i,j}s_{ij}\partial_{x_i}
\partial_{x_j}g(z)$ for all $g\in J^2$. In particular, there are
vectors $v_1,\dots,v_k$ in $\R^n$ with
$$\varphi(g)=\sum_{i=1}^kD^2g(z)[v_i,v_i]$$
for every $g\in J^2$.
Since $\varphi$ does not vanish identically on $J^2$ we have $v_i
\notin T_p(V)$ for at least one index~$i$. Therefore $\varphi(f)>0$
follows from the hypothesis.
\end{proof}

\begin{rem}\label{lcireicht}
The condition in Thm.\ \ref{innpkte} that $V$ is a local complete
intersection means that the ideal $J$ of $V$ in $\R[\x]$ can locally
be generated by $n-\dim(V)$ many elements. It is satisfied if $V$ is
nonsingular, but the condition is much more general.
\end{rem}

\begin{example}
We illustrate the use of Thm.\ \ref{innpkte} by an example. Let $M$
be an archimedean quadratic module in $\R[x,y,z]$, let $K=X(M)$, and
let $Z=\{(0,0,t)\colon t\in\R\}$ be the $z$-axis in $\R^3$. Assume
that $p$, $q$, $r\in\R[x,y,z]$ are such that
$$f\>=\>x^2\cdot p+y^2\cdot q+2xy\cdot r$$
satisfies $f>0$ on $K\setminus Z$ and $f=0$ on $Z$. Then $f\in M$,
provided that $p$ and $pq-r^2$ are strictly positive on $Z\cap K$.
This follows by a direct application of \ref{innpkte}.
\end{example}


\section*{Acknowledgments}
We are indebted to David Handelman for pointing out to us the
relevance of pure states on ideals to certificates of nonnegativity.
At the very early stages of this work, the third author appreciated
many stimulating discussions with Daniel Plaumann.


\end{document}